\documentclass[10pt,reqno]{amsart}
\usepackage{a4wide,amsmath,amscd,amsfonts,amsthm,amssymb,latexsym}
\usepackage[T1]{fontenc}
\usepackage{url}

\theoremstyle{plain}
   \newtheorem{theorem}{Theorem}
   
   \newtheorem{lemma}[theorem]{Lemma}

\theoremstyle{definition}

\theoremstyle{remark}
   \newtheorem{remark}{Remark}

\thanks{PB is a Royal Swedish Academy of Sciences Research Fellow
  supported by a grant from the Knut and Alice Wallenberg
  Foundation.}

\newcommand{\HH}{\mathcal{H}}

\newcommand{\RR}{\mathbb{R}}
\newcommand{\CC}{\mathbb{C}}

\newcommand{\cZ}{\mathcal{Z}}

\newcommand {\pr} {\prec}

\def\newop#1{\expandafter\def\csname #1\endcsname{\mathop{\rm
#1}\nolimits}}

\newop{diag}
\newop{supp}
\newop{Proj}
\newop{JP}
\newop{Hom}
\newop{Sym}
\newop{Symb}
\newop{PolP}
\newop{PolO}
\newop{glb}
\newop{MAP}
\newop{MOD}

\begin{document}




\begin{center}
\textbf{\textsc{ \Large  hyperbolicity preservers and majorization}}\\[2ex]
\textbf{\textsc{\em  \large  Pr\'eservateurs d'hyperbolicit\'e et majorisation}}\\[3ex]
{\small \sc Julius Borcea and Petter Br\"and\'en\footnote{Corresponding author}\footnote{PB is a Royal Swedish Academy of Sciences Research Fellow
  supported by a grant from the Knut and Alice Wallenberg
  Foundation.}} \\[1ex]
{\footnotesize Department of Mathematics, Stockholm University, SE-106 91 Sweden. \\ \emph{Email:} \url{pbranden@math.su.se}\\ 
\emph{Published as: } C. R. Math. Acad. Sci. Paris {\bf 348} (2010), 843--846.}\\[3ex] 

 \parbox[c]{11cm}{ \footnotesize 
 {\sc R\'esum\'e.} L'ordre de majorisation de $\RR^n$ induit un ordre partiel naturel sur
l'espace des polyn\^omes hyperboliques univari\'es de degr\'e $n$. Nous
caract\'erisons les op\'erateurs lin\'eaires sur ces polyn\^omes pr\'eservant l'ordre
donn\'e et montrons que seule la pr\'eservation de l'hyperbolicit\'e suffit (modulo
des contraintes \'evidentes sur le degr\'e).\\[4ex]
{\sc Abstract.} The majorization order on $\RR^n$ induces a natural partial ordering on the space of univariate hyperbolic polynomials of degree $n$.  We characterize all linear operators on polynomials that preserve majorization, and show that it is sufficient (modulo obvious degree constraints) to preserve hyperbolicity.  }\\[6ex]

\end{center}

\section{Introduction and main result}
A polynomial in $\RR[z]$ is \emph{hyperbolic} if it has only real zeros. 
The space $\HH_n$ of all hyperbolic polynomials of degree $n$ is equipped with a natural partial ordering defined in terms of the majorization order on weakly increasing vectors in  $\RR^n$. If $x=(x_1,\ldots, x_n)$ and $y = (y_1,\ldots, y_n)$ are weakly increasing vectors in $\RR^n$, then $y$ \emph{majorizes} $x$ (denoted $x \pr y$) if   $\sum_{i=1}^{n}x_{i}=\sum_{i=1}^{n}y_{i}$, and 
$\sum_{i=0}^{k}x_{n-i}\le \sum_{i=0}^{k}y_{n-i}$ for each $0\le k\le n-2$. Given a polynomial $p\in\HH_n$ arrange the zeros (counting multiplicities) of $p$ in a weakly increasing vector 
$
\cZ(p)\in\RR^n.
$
 If 
$p, q\in\HH_n$  we 
say that $p$ {\em is majorized by} $q$, 
denoted $p\pr q$, 
if $p$ and $q$ have the same leading coefficient and $\cZ(p)\pr \cZ(q)$.  In particular if $p \pr q$, then the top two coefficients of $p$ and $q$ are the same. The majorization order on $\HH_n$ was  studied  in 
\cite{BS,Bo1,Bo2,Per2,Per}. Particular interest has been given to the question of which linear operators on polynomials preserve majorization. 
The purpose of this note is to characterize such operators. 

Let $\RR_n[z]$ be the linear space of all real polynomials of degree at most $n$. 
A linear operator $T:\RR_n[z]\to \RR[z]$ {\em preserves majorization} 
if $T(p)\pr T(q)$ whenever $p,q\in \HH_n$ are such that $p\pr q$. Recall that two hyperbolic polynomials have \emph{interlacing zeros} if 
$$
x_1 \leq y_1 \leq x_2 \leq y_2 \leq \cdots \quad \mbox{ or } \quad y_1 \leq x_1 \leq y_2 \leq x_2 \leq \cdots, 
$$
where $x_1 \leq x_2 \leq \cdots$ and $y_1 \leq y_2 \leq \cdots$ are the zeros of $p$ and $q$, respectively.  We say that a polynomial $p(z_1,\ldots, z_n) \in \CC[z_1,\ldots, z_n]$ is \emph{stable} if it is nonzero whenever all variables have positive imaginary parts. A linear operator $T: \RR_n[z] \rightarrow \RR[z]$ is called \emph{degenerate} if $\dim(T(\RR_n[z])) \leq 2$. The \emph{symbol} of a linear operator 
$T: \RR_n[z] \rightarrow \RR[z]$ is the bivariate polynomial $F_T(z,w)=  \sum_{k=0}^n \binom n k T(z^k)w^{n-k}$. 
The following theorem is our main result and will be proved in the next section. 
\begin{theorem}\label{main}
Suppose that $T : \RR_n[z] \rightarrow \RR[z]$ is a linear operator, where $n \geq 1$. Then $T$ preserves  majorization if and only if 
\begin{enumerate}
\item $T$ is nondegenerate and $T(\HH_n) \subseteq \HH_m$ for some $m$, or 
\item $T$ is of the form $T(\sum_{k=0}^n a_k z^k)= a_nT(z^n)+a_{n-1}T(z^{n-1})$,  where  $T(z^n) \not \equiv 0$ is hyperbolic, and either $T(z^{n-1}) \equiv 0$ or $T(z^{n-1})$ is a hyperbolic polynomial which is not a constant multiple of $T(z^{n})$, and $T(z^{n-1})$  and $T(z^n)$ have interlacing zeros.  
\end{enumerate}
Moreover, condition (1) is equivalent to that $T$ is nondegenerate, and $F_T(z,w)$ or $F_T(z,-w)$  is stable and such that 
$\deg(T(z^n)) > \deg(T(z^k))$ for all $k <n$. 
\end{theorem}
Theorem \ref{main} complements \cite{BB1} where the authors characterized all linear operators on polynomials preserving hyperbolicity. 
Also, Theorem \ref{main} answers in the affirmative several questions raised in \cite{Bo1,Bo2}.

\section{Proof of Theorem \ref{main}}
We will use the algebraic characterization of hyperbolicity preservers obtained in \cite{BB1}:
\begin{theorem}\label{alg}
Suppose that $T : \RR_n[z] \rightarrow \RR[z]$ is a linear operator, where $n \geq 1$. Then $T$ preserves  hyperbolicity if and only if 
\begin{itemize}
\item $T$ is degenerate and is of the form 
$$
T(p)= \alpha(p)P+\beta(p)Q,
$$
where $\alpha,\beta : \RR_n[z] \rightarrow \RR$ are linear functionals and $P,Q$ are hyperbolic polynomials with interlacing zeros, or 
\item $T$ is nondegenerate and $F_T(z,w)$ is stable, or 
\item $T$ is nondegenerate and $F_T(z,-w)$ is stable. 
\end{itemize}
\end{theorem}

Suppose first that $T$ is degenerate. If $T$ is as in (2) of Theorem \ref{main}, then $T$ preserves hyperbolicity by Obreshkoff's theorem, see e.g. \cite[Theorem 10]{BB1}. Also, $T(p) = T(q)$ whenever $p \pr q$ which proves that (2) is sufficient. Note that if $p=\sum_{k=0}^na_kz^k \in \HH_n$, then 
$a_n(z + a_{n-1}/na_n)^n \pr p$. Hence if $T$  preserves majorization, then the degree and the top two coefficients of $T(f)$ only depend on the top two coefficients of $p$. Since $T$ is of the form $T(p)= \alpha(p)P +\beta(p)Q$, where $\alpha$ and  $\beta$ are functionals (by Theorem \ref{alg}) it is not hard to see that $T$ has to be of the form (2). Henceforth, we assume that $T$ is nondegenerate.  We start by proving that (1) is sufficient.

\begin{lemma}\label{degrees}
Suppose that $T : \RR_n[z] \rightarrow \RR[z]$ is a nondegenerate linear operator preserving hyperbolicity. Then there are numbers $0 \leq K \leq L \leq M \leq N \leq n$ such that 
\begin{enumerate}
\item $T(z^k)\equiv 0$ if   $k < K$ or $k>N$; 
\item $\deg(T(z^{k+1}))= \deg(T(z^{k}))+1$ for all $K \leq k <L$; 
\item $\deg(T(z^{k}))\leq  \deg(T(z^{L}))=\deg(T(z^{M}))$ for all $L \leq k \leq M$, and
\item $\deg(T(z^{k+1}))= \deg(T(z^{k}))-1$ for all $M \leq k <N$.
\end{enumerate} 
\end{lemma}
\begin{proof}
By Theorem \ref{alg}, either  $F_T(z,w)$ or $F_T(z,-w)$  is stable.  The lemma is a simple consequence of the fact that the support of a stable polynomial is a jump system, see \cite[Theorem~3.2]{Br1}. 
\end{proof} 
\begin{remark}\label{nd}
Suppose that $T : \RR_n[z] \rightarrow \RR[z]$ is a nondegenerate linear operator such that $T(\HH_n) \subseteq \HH_m$. Since any hyperbolic polynomial of degree at most $n$ is the limit of degree $n$ polynomials, it follows from Hurwitz' theorem on the continuity of zeros that $T$ preserves hyperbolicity. But then 
$L=M=N=n$, since otherwise one could produce two polynomials $p, q\in \HH_n$ such that $\deg(T(p)) \neq \deg(T(q))$. 
\end{remark}

 To any nondegenerate hyperbolicity preserver, we associate a sequence 
$\{\gamma_k(T)\}_{k=0}^n$ by defining $\gamma_k(T)$ to be the coefficient in front of $z^{r+k}$ in $T(z^{k})$, where $r=\deg (T(z^K))-K$ and $K$ is as in Lemma \ref{degrees}. We claim that the linear operator 
$\Gamma : \RR_n[z] \rightarrow \RR[z]$ defined by $\Gamma(z^k)= \gamma_k(T)z^k$ preserves hyperbolicity. Indeed, 
$$
\Gamma(p(z))= \lim_{\rho \rightarrow 0} (\rho/z)^r T(p(\rho z))(z/\rho),
$$
so the claim follows from Hurwitz' theorem. 
\begin{remark}\label{multsign}
It is known that such sequences have either constant sign or are alternating in sign, and that the indices $k$ for which $\gamma_k(T) \neq 0$ form an interval, see e.g. \cite[Theorem 3.4]{CC1}. 
\end{remark}

To prove Theorem \ref{main} we will  use an important result on hyperbolic polynomials in several variables. A  
homogeneous polynomial $p\in\RR[z_1,\ldots, z_n]$  is said to be {\em hyperbolic} with respect to a 
vector $e \in \RR^n$ if $p(e) \neq 0$ and for all vectors $\alpha \in \RR^n$  the polynomial $p(\alpha + et)\in\RR[t]$ has only real zeros. The following theorem,  proved by Lewis--Parrilo--Ramana based heavily on the work of Dubrovin, Helton--Vinnikov and  Vinnikov, settled the so called Lax conjecture. 

\begin{theorem}[\cite{HV,LPR}]\label{lax}
Let $p\in\RR[x,y,z]$ be a homogeneous polynomial  of degree $m$. Then $p$ is hyperbolic with respect to 
$e=(1,0,0)$  if and only if there exist two real symmetric $m \times m$ matrices $B$ and $C$  such that 
$$
p(x,y,z)=p(e)\det(xI-yB-zC).
$$
\end{theorem}

Theorem \ref{lax} enables us to use the following well known convexity result in matrix theory due to K.~Fan.  

\begin{lemma}[\cite{KF}]\label{st-k-eig}
Let $A$ be a complex Hermitian matrix of size $n \times n$, and  denote by $\lambda_1(A)\le\cdots\le \lambda_n(A)$ its 
eigenvalues arranged in weakly increasing order. For each $1\le k\le n$ the function
$$
A\mapsto \sum_{i=1}^{k}\lambda_{n+1-i}(A)
$$
is convex on the real space of Hermitian $n\times n$ matrices.
\end{lemma}

\begin{lemma}\label{st1}
Let $T:\RR_n[z]\to \RR[z]$ be a nondegenerate linear operator satisfying $T(\HH_n) \subseteq \HH_m$, where $n \geq 2$.  Let further $r(z) \in \HH_{n-2}$ be monic,  and $s$ be a fixed real number. For $t \in \RR$, let 
$x_1(t)\le\cdots\le x_m(t)$ be the zeros of the polynomial 
$T(r(z)((z+s)^2-t^2))$. Then for each $1\le k\le m$, 
\begin{equation}\label{k-f}
\RR\ni t\mapsto \sum_{i=1}^{k}x_{m+1-i}(t)
\end{equation}
is a convex and even function on $\RR$. Moreover, 
$$
T\Big(  r(z)((z+s)^2-t_1^2 )\Big)\pr T\Big(  r(z)((z+s)^2-t_2^2 )\Big),
$$ 
whenever $0\le t_1\le t_2$.

\end{lemma}

\begin{proof}
Set  
$g(z)=T(r(z)(z+s)^2)$, $h(z)=T(r(z))$, and $m=\deg g$. If $h\equiv 0$ there is nothing to 
prove so we may assume that $\deg h \ge 0$. Then 
$\deg h=m-2$ by Remark \ref{nd}. We claim that the homogeneous 
degree $m$ polynomial in 
three variables  
$$
f(z_1,z_2,z_3)=z_3^mg(z_1/z_3)-z_2^2z_3^{m-2}h(z_1/z_3)
$$
is hyperbolic with respect to the vector $e=(1,0,0)$. If $\alpha = (a,b,0)$, then 
$$
f(\alpha+et)= \gamma_m(T)(a+t)^m-b^2\gamma_{m-2}(T)(a+t)^{m-2}
$$
has only real zeros since, by Remark \ref{multsign}, $\gamma_m(T)\gamma_{m-2}(T)>0$. Also, if $\alpha = (a,b,c)$ where $c \neq 0$, then 
$$
f(\alpha+et)= c^m T\Big(r(z)(z^2-b^2/c^2)\Big)\Big|_{z=(a+t)/c}
$$
has only real zeros, and the claim follows.

By Theorem \ref{lax} 
there exist real symmetric $m\times m$ matrices $B$ and $C$ such that
$$
f(z_1,z_2,z_3)=f(e)\det(z_1I-z_2B-z_3C).
$$
It follows that for any fixed $t\in\RR$ the zeros of the polynomial
$$
T\Big(  r(z)((z+s)^2-t^2 )\Big)=f(z,t,1)=g(z)-t^2h(z)
$$
are precisely the eigenvalues of the real symmetric matrix $tB+C$. Note also that $\sum_{i=1}^m x_i(t)$ is constant in $t$, since the two top coefficients of $f(z,t,1)$ come from $g(z)$.  The lemma 
now follows from Lemma \ref{st-k-eig}.
\end{proof}

To complete the proof of the sufficiency of (1) in Theorem \ref{main} we need 
a well-known lemma  due to Hardy, Littlewood and P\'olya, see \cite{MO}. 
For simplicity, we formulate it by means of polynomials in 
$\HH_n$. Given $p,q\in\HH_n$  with $n \geq 2$, 
$\cZ(p)=(x_{1}, \ldots, x_{n})$ and 
$\cZ(q)=(y_{1},  \ldots, y_{n})$ we say that $p$ is a 
{\em pinch} of $q$ if there exist $1\le i\le n-1$ and 
$0\le t\le (y_{i+1}-y_i)/2$ such that $x_i=y_i+t$, $x_{i+1}=y_{i+1}-t$, and 
$x_k=y_k$ for $k\neq i$. Note that if $p$ is a pinch of $q$, then we may write $p$ and $q$ as 
$p(z)=r(z)((z+s)^2-t_1^2)$ and $q(z)=r(z)((z+s)^2-t_2^2)$, where $r$ is a hyperbolic polynomials and $s,t_1, t_2 \in \RR$ with $0 \leq t_1\leq t_2$.  

\begin{lemma}\label{pinch}
If $p,q\in\HH_n$, $n \geq 2$, are such that $p\pr q$, then $p$ may be obtained from 
$q$ by a finite number of pinches.
\end{lemma}
Suppose now that $p \pr q \in \HH_n$ where $n \geq 2$ and that $T$ is as in (1) of Theorem \ref{main}. By Lemma \ref{pinch} there are polynomials $p=p_0, p_1,\ldots, p_k=q$ in $\HH_n$ such that $p_{i-1}$ is a pinch of $p_i$ for all $1 \leq i \leq k$.  By Lemma \ref{st1}, $T(p_{i-1}) \pr T(p_{i})$  for all $1 \leq i \leq k$ so by transitivity $T(p) \pr T(q)$.    The case when $n=1$ follows from the case when $n=2$ by considering 
the map $T'$ defined by $T'(f)=T(f')$. 

To prove the remaining direction in Theorem  \ref{main} assume that $T$ preserves majorization.  If $\deg(T(z^n))> \deg(T(z^{n-1}))$, then by Lemma \ref{degrees},  $\deg(T(p))=\deg(T(q))$ for any two polynomials $p,q$ of degree $n$. In particular $T(\HH_n) \subseteq \HH_m$ for some $m$. Assume that $\deg(T(z^n))\leq  \deg(T(z^{n-1}))$. Recall that $\deg(T(p))$ and the top two coefficients of $T(p)$ only depend on the top two coefficients of $p$. This can only happen if $\deg(T(z^{n-2})) \leq \deg(T(z^{n-1}))-2$, since otherwise the top two coefficients of $T(z^n-a^2z^{n-2})$ would depend on the real parameter $a$. 
But then, by Lemma \ref{degrees},  $T(1) \equiv \cdots \equiv T(z^{n-2}) \equiv 0$ and $T$ is thus degenerate contrary to the assumptions. 

The final sentence in Theorem \ref{main} follows from Lemma \ref{degrees} and  Theorem \ref{alg}.

\end{document}